\documentclass[reqno,11pt]{amsart}
\usepackage{newtxtext}
\usepackage {tikz}
\usetikzlibrary{arrows,positioning}
\tikzset{
0edge/.append style={->,>=stealth', shorten <=3pt, shorten >=3pt},
1edge/.append style={->,>=stealth', dashed, shorten <=3pt, shorten >=3pt},
every loop/.append style={min distance=12mm,shorten <=3pt, shorten >=3pt}
}

\makeatletter
\def\th@plain{\slshape}
\makeatother

\newcommand{\Rbb}{\mathbb{R}}
\newcommand{\Zbb}{\mathbb{Z}}

\newcommand{\Acal}{\mathcal{A}}
\newcommand{\Gcal}{\mathcal{G}}
\newcommand{\Kcal}{\mathcal{K}}

\newcommand{\abf}{\mathbf{a}}
\newcommand{\bbf}{\mathbf{b}}

\newcommand{\p}{_{\ge0}}
\newcommand{\m}{^{-1}}

\newcommand{\ot}{\leftarrow}

\newcommand{\abs}[1]{\lvert#1\rvert}
\newcommand{\norm}[1]{\lVert#1\rVert}
\newcommand{\newword}[1]{\emph{#1}}
\newcommand{\angles}[1]{\langle #1 \rangle}
\newcommand{\set}[1]{\{ #1 \}}

\DeclareMathSymbol{\upharpoonright}{\mathrel}{AMSa}{"16}
\let\restriction\upharpoonright

\DeclareMathOperator{\GL}{GL}

\theoremstyle{plain}
\newtheorem{theorem}{Theorem}[section]
\newtheorem{lemma}[theorem]{Lemma}

\newtheorem{claim}[theorem]{Claim}

\theoremstyle{definition}
\newtheorem{definition}[theorem]{Definition}
\newtheorem{remark}[theorem]{Remark}
\newtheorem{example}[theorem]{Example}

\numberwithin{equation}{section}

\newcommand{\occurs}[3]{{}_#1 #2 {}_#3}

\begin{document}

\bibliographystyle{plain}

\sloppy

\title[Only one Farey map]{There is only one Farey map}

\author[G.~Panti]{Giovanni Panti}
\address{Department of Mathematics, Computer Science and Physics\\
University of Udine\\
via delle Scienze 206\\
33100 Udine, Italy}
\email{giovanni.panti@uniud.it}

\begin{abstract}
Let $A_0,A_1$ be nonnegative matrices in $\GL_{n+1}\Zbb$ such that the subsimplexes $A_0[\Delta],A_1[\Delta]$ split the standard unit $n$-dimensional simplex $\Delta$ in two. We prove that, for every $n=1,2,\ldots$ and up to the natural action of the symmetric group by conjugation, there are precisely three choices for the pair $(A_0,A_1)$ such that the resulting projective Iterated Function System is topologically contractive. In equivalent terms, in every dimension there exist precisely three continued fraction algorithms that assign distinct two-symbol expansions to distinct points. 
These expansions are induced by the Gauss-type map $G:\Delta\to\Delta$ with branches $A_0\m,A_1\m$, which is continuous in exactly one of these three cases, namely when it equals the Farey-M\"onkemeyer map.
\end{abstract}

\thanks{\emph{2020 Math.~Subj.~Class.}: 11J70, 37M25}

\maketitle

\section{Introduction and Preliminaries}\label{ref1}

Let $n=1,2,3,\ldots$, and let $\Delta=\set{x\in\Rbb\p^{n+1}:\sum x_i=1}$ be the standard $n$-dimensional simplex, with vertices $\set{e_0,\ldots,e_n}$. We denote the monoid of all nonnegative matrices in $\GL_{n+1}\Zbb$ by $\Sigma$; the group of invertible elements of $\Sigma$ is the symmetric group $S_{n+1}$ of permutation matrices.
Every $A\in\Sigma$ determines a continuous map of $\Delta$ into itself, still denoted by $A$, via $A(x)=Ax/\norm{Ax}_1$; the image $A[\Delta]$ is a \newword{unimodular simplex}.
In agreement with the numbering of vertices of $\Delta$, we number matrix rows and columns from $0$ to $n$.

Let $\Acal=\set{A_0,\ldots,A_{m-1}}$ be a nonempty finite subset of $\Sigma$ not containing invertible elements. Identifying matrices with maps as above, $\Acal$ is an \newword{Iterated Function System} (IFS) on $\Delta$; all IFS in this paper are of this form.
We let $a,b,\ldots$ vary in the alphabet $\set{0,\ldots,m-1}$ and let $v,w,\ldots$ (respectively $\abf,\bbf,\ldots$) stand for finite words (respectively infinite sequences) over that alphabet. We write $\abf\restriction t=a_0\ldots a_{t-1}$ for the $t$-long prefix of $\abf$, and let $A_{\abf\restriction t}$ and $\Delta_{\abf\restriction t}$ stand for the product $A_{a_0}\cdots A_{a_{t-1}}$ and for the unimodular simplex determined by that product, respectively.

We say that $\Acal$ is \newword{simplex-splitting} if $\Delta_0,\ldots,\Delta_{m-1}$ constitute a partition of $\Delta$, that is, $\Delta=\bigcup_a \Delta_a$ and the interiors of $\Delta_a$ and $\Delta_b$ do not intersect for $a\not=b$. Given $x\in\Delta$, let $a=a(x)$ be the least index such that $x\in\Delta_a$ (the choice of this specific  selection rule is irrelevant; we will readdress this issue in Remark~\ref{ref13}(5)). Then these data define the Gauss-type map $G:\Delta\to\Delta$ by $Gx=A_{a(x)}\m(x)$, and the sequence $a_0 a_1 \ldots$ given by $a_t=a(G^tx)$ is the \newword{digit expansion} of $x$.

\begin{definition}
The\label{ref3} IFS $\Acal$ is \newword{topologically contractive} if for every infinite sequence~$\abf$ the intersection $\bigcap_{t\ge0}\Delta_{\abf\restriction t}$ is a singleton.
\end{definition}

\begin{remark}\label{ref13}.
\begin{enumerate}
\item In the classical setting of IFS, one usually requires \newword{uniform contractivity}, that is, the existence of a number $r\in(0,1)$ and a metric $d$ such that, for every $A\in\Acal$ and every pair of points $x,y$, we have $d\bigl(A(x),A(y)\bigr)\le r\,d(x,y)$. Uniform contractivity is strictly stronger than topological contractivity. Since we never use uniform contractivity, for simplicity we drop the adjective ``topological''. In the literature on multidimensional continued fractions a related concept is \newword{topological convergence}, which means contractivity \emph{along sequences that code a $G$-orbit}~\cite[Definition~9]{schweiger00}.
\item Contractivity implies that the function $\pi:m^\omega\to\Delta$ that maps $\abf$ to the only element of the singleton in Definition~\ref{ref3} is well defined; clearly, $\pi$ is then continuous.
\item In dimension $n=1$ every $\Acal$ is contractive; see Theorem~\ref{ref4}.
\item If $\Acal$ is simplex-splitting and contractive, then 
every point of~$\Delta$ has an expansion, and
distinct points have distinct expansions.
\item If $\Acal$ is simplex-splitting then the matrices in $\Acal$ constitute a \newword{code} inside $\Sigma$ (that is, they generate a free submonoid); this follows from the Ping-Pong Lemma~\cite[VII.A.2]{delaharpe00}, \cite[Lemma~2.5]{panti24}. If, moreover, $\Acal$ is contractive, then in the definition of $G$ we may forget about the selection rule and treat $G$ as a multivalued map, sending $x$ to the set $\set{A_a\m(x):x \in \Delta_a}$. This is safe because~\cite[Theorem~2.13]{panti24} (actually, a simplified version of it for ordinary IFS, not graph-directed ones) guarantees that the number of images remains uniformly bounded: that is, there exists a bound $M$ (depending on $\Acal$), such that for every $x$ and every~$t$ the cardinality of set of images of $x$ at time $t$ is bounded by $M$. Moreover, if $x$ is periodic (that is, $x$ belongs to the set of its own images at some time $t>0$), then $G$ is an ordinary single-valued map along the orbit of $x$.
\item Contractivity also implies that the \newword{Hutchinson operator} $\Acal(K)=\bigcup_a A_a[K]$ on the compact space $\Kcal$ of nonempty compact subsets of $\Delta$ is a strict contraction with respect to the Hausdorff metric 
\[
d(J,K)=\inf\set{\varepsilon>0:J\subseteq K_{<\varepsilon}\text{ and }
K\subseteq J_{<\varepsilon}},
\]
where $K_{<\varepsilon}$ is the union of all open balls of radius $\varepsilon$ and center a point of~$K$.
Therefore the Hutchinson operator has a unique fixed point, namely $\pi[m^\omega]$, and this fixed point is attractive; see the discussion in~\cite[\S1]{banakh_et_al15}.
\end{enumerate}
\end{remark}

\begin{example}
Let\label{ref5} $\Acal$ be the set of the following three matrices:
\[
A_0=\begin{pmatrix}
1 & &\\
1 & 1 & \\
1 & & 1
\end{pmatrix},\quad
A_1=\begin{pmatrix}
1 &1 &\\
 & 1 & \\
1 &1 & 1
\end{pmatrix},\quad
A_2=\begin{pmatrix}
1 & & 1\\
 & 1 & 1\\
 & & 1
\end{pmatrix}.
\]
Then $\Acal$ is simplex-splitting; Figure~\ref{fig1} left shows the time-$5$ partition $\set{\Delta_w:\abs{w}=5}$.
\begin{figure}[ht]
\includegraphics[width=6.2cm]{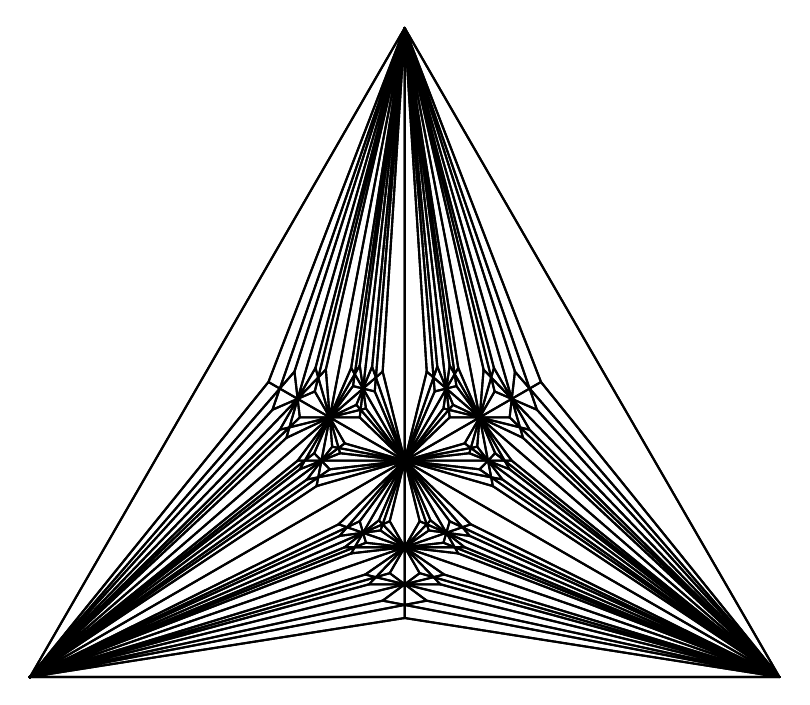}
\hspace{0.0cm}
\includegraphics[width=6.2cm]{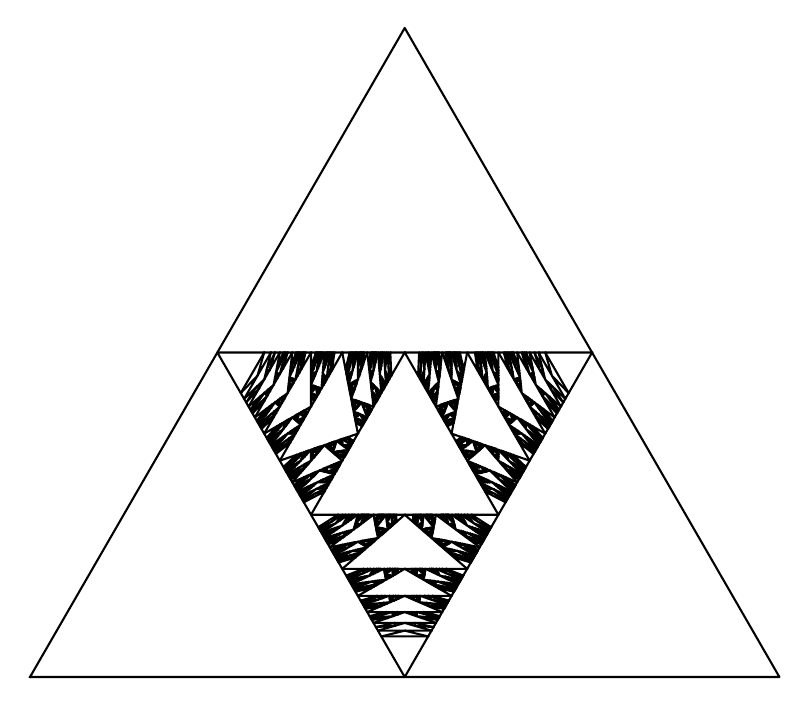}
\caption{The time-$5$ partition and the least fixed point $R$}
\label{fig1}
\end{figure}

The Hutchinson operator has a largest fixed point, namely $\Delta$ itself, and a smallest one $R$, which is the projection to $\Delta$ of the closure of the set of the attracting fixed points in projective $2$-space of the proximal elements in the monoid generated by $A_0,A_1,A_2$~\cite[\S3.1]{benoist-quint16}. It is easy to show that $R$ is the image of the Rauzy gasket~\cite{avila-hubert-skripchenko16}, \cite{jurgarauzy} under an appropriate projective map. There are uncountably many incomparable fixed points between $\Delta$ and $R$; for example, for every $x$ in the boundary of $\Delta$ the sequence $\Acal^t\set{x}$ converges to the closure of $\set{A_w(x):\text{$w$ is a word in $\set{0,1,2}$}}$, which is a fixed point intersecting the boundary of $\Delta$ in $x$ only.

The map $G$ acts on $x=(x_0,x_1,x_2)$ by subtracting the smallest coordinate from the other coordinates and projecting back to $\Delta$. Ambiguities arise when the smallest coordinate is not unique; they are irrelevant but for the case of the vertices, that in the absence of a selection rule would have uncountably many digit expansions.
Contractivity is violated uncountably many times, for example at any sequence with a tail avoiding precisely one letter. Take, for example, $\abf=(01)^\omega$;
by looking at the right eigenspaces of $A_{01}$ one readily sees that $\Delta_{\abf\restriction t}$ converges to the line segment
of vertices $e_2$ and
 $(-\tau/2+1,\tau/2-1/2,1/2)$, with $\tau$ the golden ratio. Therefore, all points in that line segment have $\abf$ as their digit expansion. This way of violating contractivity will reappear as a special case of Lemma~\ref{ref12}.
\end{example}

In dimension greater than $1$ simplex-splitting contractive algorithms are scarce; indeed, some experimenting shows that a ``random'' choice of a splitting of $\Delta$ has little probability of being contractive, even if one takes care that the sequences~$\Delta_{a^t}$ determined by the generators always
shrink to a point.
Our main result, Theorem~\ref{ref6}, shows an extreme case of this scarcity: \emph{regardless of the dimension, and modulo the obvious symmetries of the unit simplex, there are only three contractive simplex-splitting IFS on two maps}.

We conclude this preliminary section by showing, on the positive side, that the arithmetic of $\Sigma$ always forces convergence to lower-dimensional simplexes.

\begin{theorem}
Let\label{ref4} $\Delta=K_0\supseteq K_1\supseteq K_2\supseteq\cdots$ be any descending chain of simplexes; then the intersection $K=\bigcap_t K_t$ is a simplex. If, in addition, each $K_t$ is unimodular and the chain is strictly decreasing infinitely often, then $K$ has dimension strictly less than $n$.
\end{theorem}
\begin{proof}
It is easy to show directly (see the proof of~\cite[Lemma A.2.3]{bishop-peres17} for a more general statement) that $K$ is the limit of the sequence $K_t$ in the Hausdorff metric.
Let $\Kcal'$ be the subspace of $\Kcal$ whose elements are all nonempty subsets of $\Delta$ of cardinality at most $n+1$. Writing $V_t$ for the set of vertices of $K_t$, the sequence $V_0,V_1,\ldots$ has an accumulation point $V=\set{x_1,\ldots,x_k}$ in $\Kcal'$, because the latter is compact (it is the quotient of the product of $n+1$ copies of $\Delta$ under the action of the symmetric group, see~\cite[p.~877]{borsuk-ulam31}).
Now, the map that sends an element of $\Kcal$ to its convex hull is Lipschitz~\cite[p.~184]{federer69}, thus continuous. This implies that the convex hull of $V$, which is a simplex, equals $K$, thus establishing our first claim.

Assume now that every $K_t$ is unimodular and that, without loss of generality, strictly contains $K_{t+1}$. This implies that there exists a sequence $A_1,A_2,\ldots$ of matrices in $\Sigma$, none of them a permutation matrix, such that $K_t=A_1A_2\cdots A_t[\Delta]$.
Letting $\lambda$ denote Lebesgue probability on $\Delta$, we conclude the proof by showing that $\lambda(K)=0$, so that $K$ is not $n$-dimensional. Consider the row vector $(1\ldots1)A_1\cdots A_t=\bigl(l_0(t)\ldots l_n(t)\bigr)$; it is known~\cite[p.~388]{nogueira95} that $\lambda(K_t)=\bigl(l_0(t)\cdots l_n(t)\bigr)\m$, thus the reciprocal of an integer. Since no matrix in the sequence is a permutation matrix, $\lambda(K_{t+1})$ is strictly less than $\lambda(K_t)$, and therefore is the reciprocal of a strictly greater integer. This holds for every $t$, and hence $\lambda(K)=\lim_t\lambda(K_t)=0$.
\end{proof}

\section{Coding by two symbols}\label{ref7}

From now on we will consider only simplex-splitting algorithms on two symbols, $\Acal=\set{A_0,A_1}$. We first discuss the number of these IFS modulo the natural symmetries of $\Delta$; this extends the discussion in~\cite{garrity-osterman240000}, that treats the $2$-dimensional case.
We need to define a few elementary matrices in $\Sigma$:
\begin{itemize}
\item $N$ is the $(n+1)\times(n+1)$ elementary matrix obtained from the identity matrix $I$ by summing the first column to the second;
\item $F$ is obtained from $I$ by exchanging the first two columns;
\item $R$ is obtained from $I$ by shifting cyclically each column to the previous one, while moving the first to the last;
\item $L=FNF$.
\end{itemize}

Obviously the only way of splitting $\Delta$ in two unimodular subsimplexes is by choosing an edge $\angles{e_i,e_j}$ and taking as first subsimplex the one with vertices $\set{(e_i+e_j)/2}\cup\set{e_k:k\not=j}$, and as second the one with vertices $\set{(e_i+e_j)/2}\cup\set{e_k:k\not=i}$. Without loss of generality, that is up to the group $S_{n+1}=\angles{F,R}\subset\Sigma$ of symmetries of $\Delta$, we choose the edge $\angles{e_0,e_1}$, so that the first subsimplex has vertices the columns of $N$, and the second those of $L$.
We agree that the matrix $A_0$ maps $\Delta$ to the first subsimplex, and $A_1$ to the second. Thus, from now on our IFS are given by \emph{ordered} pairs $(A_0,A_1)$ of the form $(NP_0,LP_1)$, for arbitrary pairs $(P_0,P_1)\in S_{n+1}^2$. Two of these pairs must be identified if conjugated by an element of the subgroup $S_{n-1}$ of $S_{n+1}$ whose elements $H$ fix
both $e_0$ and $e_1$, and also if conjugated by $F$ provided that, in this second case, we exchange $A_0$ with $A_1$. Note that conjugation by elements of $S_{n-1}$ fixes both $N$ and $L$, while conjugation by $F$ exchanges them. Letting $S_2=\set{I,F}$, this means that we must count the number of orbits of pairs $(P_0,P_1)\in S_{n+1}^2$ under the action of $S_2\times S_{n-1}$ given by
\begin{equation}\label{eq1}
\begin{split}
(I,H)*(P_0,P_1)&=(HP_0H\m,HP_1H\m),\\
(F,I)*(P_0,P_1)&=(FP_1F\m,FP_0F\m).
\end{split}
\end{equation}

For $n=1,2$ the counting can be done by hand. There are three orbits for $n=1$, namely $\bigl\{(I,I)\bigr\}$, $\bigl\{(I,F),(F,I)\bigr\}$, $\bigl\{(F,F)\bigr\}$, and $21$ orbits for $n=2$, which have been computed in~\cite[\S7]{garrity-osterman240000}.
For larger $n$ the number of orbits grows fast; we do not know of a closed formula for this number as a function of $n$, nor expect one. However, the following snippet of SageMath code counts orbits for small values of~$n$; for $n=3,4,5$ there are $160$, $1283$, $11321$ orbits, respectively.
\begin{verbatim}
n=5
S=SymmetricGroup([0..n])
f=S("(0,1)")
if n<3:
    H=S.subgroup([S.identity()])
else:
    H=S.subgroup([S("(2,3)"),S(str(tuple([2..n])))])
def action(x,y):
    Hxy=[(h*x*h^-1,h*y*h^-1) for h in H]
    return Set(Hxy+[(f*o[1]*f^-1,f*o[0]*f^-1) for o in Hxy])
orbits=Set([action(p0,p1) for p0 in S for p1 in S])
print(len(orbits))
\end{verbatim}

In~\cite[Satz~10]{monkemeyer54} M\"onkemeyer proved ---using, of course, a different language--- that the IFS $(M_0,M_1)$, where $M_0=NFR$ and $M_1=LR$, is contractive in every dimension~$n$; a more modern proof is in~\cite[Lemma~19]{schweiger00}. M\"onkemeyer's algorithm has been rediscovered several times, notably by Selmer~\cite{selmer61} and Baldwin~\cite{baldwin92a}: see the discussion at the end of~\cite[\S1]{panti08}. In dimension~$1$, and identifying $\Delta$ with the real unit interval $[0,1]$ via $(x,1-x)\mapsto 1-x$, it reduces to the classical Farey algorithm~\cite{ito89}, \cite{kessebohmerstratmann07}, \cite{isola11}, \cite{heersink16}. The following is our main result.

\begin{theorem}
For\label{ref6} every $n=1,2,3,\ldots$, and up to the natural action of the symmetric group $S_{n+1}$, there are precisely three
algorithms that split the unit $n$-dimensional simplex $\Delta$ in two, and are contractive.
\begin{enumerate}
\item The M\"onkemeyer algorithm, induced by $(M_0,M_1)$; this is the only case in which the associated Gauss-type map $G$ is continuous. If $n$ is odd, $G$ is orientation-preserving on $\Delta_0$ and orientation-reversing on $\Delta_1$; this behavior is reversed for even $n$, and also reversed by replacing $(M_0,M_1)$ with the equivalent pair $(FM_1F\m,FM_0F\m)$.
\item The algorithm induced by $(M_0,M_1F)$, whose map $G$ is orientation-preserving for odd $n$ and orientation-reversing for even $n$.
\item The algorithm induced by $(M_0F,M_1)$, whose map $G$ is orientation-reversing for odd $n$ and orientation-preserving for even $n$.
\end{enumerate}
\end{theorem}

We remark that the orientation-preserving algorithm is the one used in~\cite{fougeron-skripchenko21} and, in an accelerated form, in~\cite{bruin-troubetzkoy03}, while the orientation-reversing one is the Cassaigne algorithm in~\cite[\S8]{fogg-nous24}. The rest of this paper is devoted to the proof of Theorem~\ref{ref6}.

\section{Proof of Theorem~\ref{ref6}}\label{ref8}

The orientation preserving/reversing claims amount to the fact that the matrices determining the branches of $G$ have positive/negative determinant. They are easily checked by noting that $\det N=\det L=1$, $\det F=-1$, and $\det R=(-1)^n$. In Case (1) the map $G$ is continuous because, for every $i\not=0$, we have $M_0e_i=M_1e_i$, 
and $\Delta_0\cap\Delta_1$ is the common $M_0$-image and $M_1$-image of the face of $\Delta$ with vertices $e_1,\ldots,e_n$.
In Cases~(2) and~(3), if $G$ were continuous then, since its two branches are coherent with respect to orientation, the images $G[\Delta_0]$ and $G[\Delta_1]$ would lie on opposite sides of some $(n-1)$-dimensional face of $\Delta$, which is impossible.

We have $FM_0=M_1$ and $FM_1=M_0$.
Therefore, every product of $t$ matrices from $(M_0,M_1F)$, or from $(M_0F,M_1)$, equals a product of $t$ matrices from $(M_0,M_1)$, possibly followed by a final $F$. Since $F[\Delta]=\Delta$, this implies that:
\begin{itemize}
\item[(i)] The time-$t$ partition $\set{\Delta_w:\abs{w}=t}$ is unchanged by taking $(A_0,A_1)$ to be either $(M_0,M_1)$, or $(M_0,M_1F)$, or $(M_0F,M_1)$.
\item[(ii)] The contractivity of the M\"onkemeyer algorithm quoted above extends to its orientation-preserving and orientation-reversing versions.
\end{itemize}

Looking back at the properties of the action~\eqref{eq1}, the proof of Theorem~\ref{ref6} reduces then to establishing the following claim.

\begin{claim}
Let\label{ref10} $P_0,P_1\in S_{n+1}$ be such that $(A_0,A_1)=(NP_0,LP_1)$ is contractive. Then, under the action of $S_2\times S_{n-1}$ given by
\begin{equation}\label{eq2}
\begin{split}
(I,H)*(A_0,A_1)&=(HA_0H\m,HA_1H\m),\\
(F,I)*(A_0,A_1)&=(FA_1F\m,FA_0F\m).
\end{split}
\end{equation}
the pair $(A_0,A_1)$ is equivalent to precisely one of $(M_0,M_1)$, $(M_0,M_1F)$, $(M_0F,M_1)$ ($(M_0F,M_1F)$ is equivalent to $(M_0,M_1)$, since they are mapped to each other by $(F,I)$).
\end{claim}

We assume familiarity with the basics of nonnegative matrices: reducible and irreducible matrices, their incidence graphs, the period of an irreducible matrix, the Perron-Frobenius theorem; see, for example, \cite[Chapter III]{gantmacher59} or~\cite[Chapter~2]{berman-plemmons79}.
The following lemma and its corollary Lemma~\ref{ref14} are our main tool.

\begin{lemma}
Let\label{ref12} $C,D\in\Sigma$ be such that:
\begin{itemize}
\item[(i)] there exists a vertex $e_i$ of $\Delta$ which is fixed by both $C$ and $D$;
\item[(ii)] the interiors of $C[\Delta]$ and $D[\Delta]$ do not intersect.
\end{itemize}
Then there exists a product $E$ of $C$ and $D$ such that $E^t[\Delta]$ does not shrink to a point, for $t\to\infty$.
\end{lemma}
\begin{proof}
It is enough to prove that the monoid $\Gamma$ generated by $C$ and $D$ contains a matrix $E$ with spectral radius strictly greater than~$1$. Indeed, if so, then the spectral radius is an eigenvalue of $E$, corresponding to an eigenvector $v\in\Delta$. Since clearly $v\not= e_i$, we have that $E^t[\Delta]$ does not shrink to a point.

Assume by contradiction that all matrices in $\Gamma$ have spectral radius~$1$. By~\cite[\S6]{protasov-voynov17} we may conjugate by an appropriate permutation matrix and assume without loss of generality that all matrices in $\Gamma$ have block upper-triangular form
\[
\begin{pmatrix}
B_1 & * & *\\
&\ddots & *\\
& & B_r
\end{pmatrix},
\]
and, moreover, the monoids $\Gamma_1,\ldots,\Gamma_r$ whose elements are the blocks appearing on the diagonal in position, respectively, $1,\ldots,r$ are all irreducible by permutations (that is, none of them admits a proper nontrivial invariant coordinate subspace). 
Clearly, every matrix in every $\Gamma_i$ still has spectral radius~$1$. By~\cite[Proposition~12]{protasov-voynov17}, every $\Gamma_i$ is a finite semigroup, thus a group of permutation matrices.

This implies that appropriate powers $C^s, D^t$ of $C,D$ are upper-triangular, with~$1$ along the diagonal; let $T=D^{-t}C^s$, which has the same form. It is easy to see that there exists a positive (that is, all entries are strictly positive) vector $x$ such that $y=Tx$ is positive as well. Indeed, starting from any $x_n>0$, one simply chooses recursively $x_k$ so large that $x_k+T^k_{k+1}x_{k+1}+\cdots+T^k_n x_n>0$. 
Thus, the point $C^sx=D^ty$ belongs, once projected to $\Delta$, to the intersection of the interiors of $C^s[\Delta]$ and $D^t[\Delta]$, which is a subset of the intersection of the interiors of $C[\Delta]$ and $D[\Delta]$; this contradicts assumption~(ii).
\end{proof}

Let now $(A_0,A_1)$ be as in Claim~\ref{ref10}; we begin by examining~$A_0$. Every column of $A_0$ contains precisely one $1$, with the exception of one column, say column $h$, that contains two $1$ in the first two rows. We construct the \newword{incomplete incidence graph} $\Gcal_0^-$ of $A_0$, having nodes $0,1,\ldots,n$ and a directed edge $j\ot i$ whenever $j\not=h$ and the entry $(A_0)^i_j$ in row $i$ and column $j$ equals $1$. We thus have $A_0(e_j)=e_i$ (this seemingly unnatural direction of arrows is the common convention in the theory of graph-directed IFS).
The incomplete graph is \newword{completed} by adding the two edges $0\to h\ot 1$, thus obtaining the ordinary incidence graph $\Gcal_0$ of $A_0$; note that $A_0$ maps $e_h$ neither to $e_0$ nor to $e_1$, but to their Farey sum $(e_0+e_1)/2$.
We let $\Gcal_1^-$, $\Gcal_1$ be the analogously defined incomplete and complete incidence graphs of $A_1$; swapping the nodes $0$ and $1$, all considerations above apply.

We now glue together $\Gcal_0^-$ and $\Gcal_1^-$, forming the \newword{incomplete graph} $\Gcal^-$ of the pair $(A_0,A_1)$, with nodes $0,\ldots,n$ and \newword{$0$-edges} and \newword{$1$-edges}, coming from $\Gcal_0^-$ and $\Gcal_1^-$, respectively. For every word $v=a_0\ldots a_{t-1}\in\set{0,1}^t$ and every pair of nodes $i,j$ there is at most one path in $\Gcal^-$ that starts from $i$ and follows first an $a_0$-edge $i_1\ot i$, then an $a_1$-edge $i_2\ot i_1$ and so on, reaching $j$ after the final $a_{t-1}$-edge.
If such a path exists we say that $\occurs{i}{v}{j}$ \newword{occurs} in $\Gcal^-$. If $\occurs{i}{v}{i}$ occurs, then 
it is a \newword{loop for $v$ at $i$}. Words $v,w$ are \newword{incomparable} if neither of them is an initial segment of the other.

\begin{lemma}
Let\label{ref14} $\Gcal^-$ be the incomplete graph associated to the contractive pair $(A_0,A_1)$. Then $\Gcal^-$ cannot contain:
\begin{itemize}
\item[(i)] either two loops for the same word at distinct nodes;
\item[(ii)] or two loops for incomparable words at the same node.
\end{itemize}
\end{lemma}
\begin{proof}
If there are two loops, at $i$ and $j$, for the word $v$, then $A_v$ fixes both $e_i$ and $e_j$ which are hence both in $\bigcap_{t\ge0}A_v^t[\Delta]$; this contradicts contractivity.

Let $v$ and $w$ be incomparable. Then it is easy to see (for example by factoring out the longest common prefix) that the interiors of $A_v[\Delta]$ and $A_w[\Delta]$ do not intersect. On the other hand, if they determine loops in $\Gcal^-$ at the same node $i$, then both $A_v$ and $A_w$ fix the vertex $e_i$. By Lemma~\ref{ref12} this again contradicts contractivity.
\end{proof}

\begin{example}
As\label{ref17} a clarifying example we draw the incomplete graphs of the M\"onkemeyer matrix pairs; $0$-edges are continuous and $1$-edges dashed. We use $n=4$; larger values just require adding the obvious tail.
\begin{enumerate}
\item The \newword{parabolic-hyperbolic} pair $(M_0,M_1)$ has incomplete graph~$\Gcal^-$
\begin{figure}[h!]
\begin{tikzpicture}[scale=0.8]
\node (0) at (0,2)  []  {$0$};
\node (1) at (0,0) []  {$1$};
\node (2)  at (2,1) []  {$2$};
\node (3)  at (4,1) []  {$3$};
\node (4)  at (6,1) []  {$4$};
\path
(0) edge[0edge,out=150,in=-150,loop] (0)
(4) edge[0edge,bend right] (3)
(3) edge[0edge,bend right] (2)
(2) edge[0edge,bend right] (1)

(4) edge[1edge,bend left] (3)
(3) edge[1edge,bend left] (2)
(2) edge[1edge,bend left] (1)
(1) edge[1edge] (0);
\end{tikzpicture}
\end{figure}

\item The \newword{parabolic-parabolic} pair $(M_0,M_1F)$ has incomplete graph~$\Gcal^-$
\begin{figure}[h!]
\begin{tikzpicture}[scale=0.8]
\node (0) at (0,2)  []  {$0$};
\node (1) at (0,0) []  {$1$};
\node (2)  at (2,1) []  {$2$};
\node (3)  at (4,1) []  {$3$};
\node (4)  at (6,1) []  {$4$};
\path
(0) edge[0edge,out=150,in=-150,loop] (0)
(4) edge[0edge,bend right] (3)
(3) edge[0edge,bend right] (2)
(2) edge[0edge] (1)

(4) edge[1edge,bend left] (3)
(3) edge[1edge,bend left] (2)
(2) edge[1edge] (0)
(1) edge[1edge,out=-150,in=150,loop] (1);
\end{tikzpicture}
\end{figure}

\item The \newword{hyperbolic-hyperbolic} pair $(M_0F,M_1)$ has incomplete graph~$\Gcal^-$
\begin{figure}[h!]
\begin{tikzpicture}[scale=0.8]
\node (0) at (0,2)  []  {$0$};
\node (1) at (0,0) []  {$1$};
\node (2)  at (2,1) []  {$2$};
\node (3)  at (4,1) []  {$3$};
\node (4)  at (6,1) []  {$4$};
\path
(4) edge[0edge,bend right] (3)
(3) edge[0edge,bend right] (2)
(2) edge[0edge] (0)
(0) edge[0edge,bend right] (1)

(4) edge[1edge,bend left] (3)
(3) edge[1edge,bend left] (2)
(2) edge[1edge] (1)
(1) edge[1edge,bend right] (0);
\end{tikzpicture}
\end{figure}

\end{enumerate}

The adjectives parabolic/hyperbolic are justified. Indeed, $M_0$ (and its $F$-conjugate $M_1F$) has characteristic polynomial $(x-1)(x^n-1)$, and $M_0^t[\Delta]$ shrinks to the vertex~$e_0$. On the other hand, the completed graph of $M_1$ is given by the full loop $n\ot 0\ot 1\ot 2\ot\cdots\ot n$, which is shortcut by the further edge $n\ot 1$.
Therefore $M_1$ is irreducible of period~$\gcd(n+1,n)=1$, its powers are eventually strictly positive, and $M_1^t[\Delta]$ shrinks to a point in the interior of $\Delta$; moreover, its characteristic polynomial $x^{n+1}-x-1$ is irreducible~\cite[Theorem~1]{selmer56}.
\end{example}

As the $1$-dimensional case is trivial, from now on we assume $n\ge2$. By construction, 
every node in $\Gcal_0^-$ has precisely one father and one son, with the exception of~$h$ that has no father and of $1$ that has no son. 
We must have $h\not=1$, since otherwise $\Gcal_0^-$ would contain a loop with at least two nodes, or two one-node loops, contrary to Lemma~\ref{ref14}.
Therefore $\Gcal_0^-$
decomposes in a chain from $h$ to $1$ and at most one loop $i\ot i$; if so,
necessarily $i=0$, for otherwise the completed graph $\Gcal_0$ would be disconnected, and $A_0^t[\Delta]$ would not shrink to a point.
We relabel the nodes in $\set{2,3,\ldots,n}$ in such a way that every $i\ge3$ is an ancestor of $i-1$ in $\Gcal_0^-$. This amounts to replacing $(A_0,A_1)$ with $(HA_0H\m,HA_1H\m)$ ---which for simplicity we rename $(A_0,A_1)$--- for an appropriate permutation matrix $H\in S_{n-1}$.
Upon exchanging $0$ with $1$, the analysis above holds for $\Gcal_1^-$, except that the nodes $2,\ldots,n$ appear along the chain of $\Gcal_1^-$ in some fixed but up to now unknown order $k_2,\ldots,k_n$.

The presence or absence of an isolated one-node loop distinguishes the parabolic from the hyperbolic cases;
we summarize the situation thus far.

\noindent $\Gcal_0^-$ is:
\begin{itemize}
\item[(0p)] either the $n$-node chain $1\ot 2\ot\cdots\ot n$ plus the loop $0\ot 0$;
\item[(0h)] or an $(n+1)$-node chain $1\ot 2\ot \cdots\ot 0\ot \cdots\ot n$, with $0$ appearing in some up to now indeterminate position (but different from the last, which is taken by $1$).
\end{itemize}

\noindent $\Gcal_1^-$ is:
\begin{itemize}
\item[(1p)] either the $n$-node chain $0\ot k_2\ot\cdots\ot k_n$ plus the loop $1\ot 1$;
\item[(1h)] or an $(n+1)$-node chain $0\ot k_2\ot \cdots\ot 1\ot \cdots\ot k_n$, with $1$ appearing in some up to now indeterminate position (but different from the last, which is taken by $0$).
\end{itemize}

\begin{lemma}
If\label{ref15} we are in Case~(0h), then the chain $\Gcal_0^-$ ends with $1\ot 0\ot 2$.
Analogously, if we are in case  Case~(1h) the chain $\Gcal_1^-$ ends with $0\ot 1\ot k_2$.
\end{lemma}
\begin{proof}
Assume we are in Case~(0h); then there cannot be two nodes $i,j$, in the chain in $\Gcal_1^-$ and such that, in $\Gcal_0^-$, the node $0$ is an ancestor of $i$, and $i$ an ancestor of $j$. Indeed, if so, then $\occurs{0}{0^r}{i}$
and $\occurs{i}{0^s}{j}$ both occur in $\Gcal^-$. for certain $r,s\ge1$.
Moreover, there exist $t,u$ such that either $\occurs{j}{1^t}{i}$ and $\occurs{i}{1^u}{0}$ occur, or $\occurs{i}{1^t}{j}$ and $\occurs{j}{1^u}{0}$ occur. In the first case the incomparable words $0^s1^t$ and $1^u0^r$ are loops at $i$, and in the second so are the words $0^s1^u0^r$ and $1^t1^u0^r$, contradicting Lemma~\ref{ref14}(ii).
This establishes Lemma~\ref{ref15} whenever Case~(1h) holds, or Case~(1p) holds and $3$ is a descendant of~$0$ in $\Gcal_0^-$.

We have now to prove that the assumption of Cases~(0h) and (1p), and of the fact that the $0$-chain ends with $1\ot 2\ot 0$, leads to a contradiction.

If $k_2\not=2$, then for certain $r,s$ both $\occurs{{k_2}}{0^r}{0}$ and $\occurs{2}{1^s}{{k_2}}$ occur, creating the incomparable loops $0^r01^s$ and $101^s$ at $k_2$, which is impossible.

Finally, $k_2=2$ implies that $\set{k_3\ldots,k_n}$ equals $\set{3,\ldots,n}$ as sets. The $1$-chain $k_3\ot\cdots\ot k_n$ must run through this set of nodes in some order. However, this $1$-chain cannot contain any \newword{backtrack}, that is, any edge $i\to j$ with $i<j$.
Indeed:
\begin{itemize}
\item if $j=i+1$, this would create two loops for the word $01$, at the distinct nodes $0$ and $j$;
\item if $j>i+1$, then $i+1$ would be either a $1$-ancestor of $i$, or a $1$-descendant of $j$. In both cases, this would create two incomparable loops at $j$.
\end{itemize}
Thus, the absence of backtracks forces $\Gcal^-$ to appear as follows:
\begin{figure}[h!]
\begin{tikzpicture}[scale=0.8]
\node (1) at (0,0)  []  {$1$};
\node (2) at (2,0)  []  {$2$};
\node (0) at (4,0)  []  {$0$};
\node (3)  at (6,0) []  {$3$};
\node (4)  at (8,0) []  {$4$};
\path
(2) edge[0edge] (1)
(0) edge[0edge] (2)
(3) edge[0edge] (0)
(4) edge[0edge] (3)

(1) edge[1edge,out=150,in=-150,loop] (1)
(2) edge[1edge,bend right] (0)
(3) edge[1edge,bend right] (2)
(4) edge[1edge,bend right] (3);
\end{tikzpicture}
\end{figure}

\noindent As in Example~\ref{ref17}, continuous and dashed edges are $0$- and $1$-edges, respectively.
We set $n=4$ for concreteness: add the obvious tail for larger $n$, and remove nodes in excess and all edges related to them for $n=2,3$.
It is however an appropriate choice, for $n$ has to be even. Indeed, by completing $\Gcal_0^-$, that is adding the edges $0\to n\ot 1$, we obtain a loop with $n+1$ edges, which is shortcut by a loop of $n-1$ edges. 
Thus the period $\gcd(n+1,n-1)$ of the irreducible matrix $A_0$ is $1$ if and only if $n$ is even; if so, $A_0^t[\Delta]$ shrinks to a point in the interior of $\Delta$. Otherwise the period is~$2$, and $A_0^2$ has nontrivial block-diagonal form, preventing contractivity.

We complete our quest for a contradiction by showing that $A_{01}^t[\Delta]$ does not shrink to a point. If $n=2$ this is established by computing the right eigenspaces of
\[
A_{01}=\begin{pmatrix}
1 & & 1\\
 & & 1\\
 & 1 & 1
\end{pmatrix}.
\]
If $n=4,6,8,\ldots$ we see by direct inspection that the incidence graph of $A_{01}$ has the following form, with a loop of length $n$ shortcut to one of length $n/2$.
\begin{figure}[h!]
\begin{tikzpicture}[scale=1.1]
\node (0) at (-4,0)  []  {$0$};
\node (nm1) at  (-2,0)  []  {$n-1$};
\node (nm3) at (-1.5,-1.5)  []  {$n-3$};
\node (nm5)  at (0,-2) []  {$n-5$};
\node (3)  at (1.5,-1.5) []  {$3$};
\node (n) at (2,0)  []  {$n$};
\node (nm2) at (1.5,1.5)  []  {$n-2$};
\node (2) at (0,2)  []  {$2$};
\node (1)  at (-1.5,1.5) []  {$1$};
\path
(0) edge[0edge,out=150,in=-150,loop] (0)
(0) edge[0edge] (nm1)
(nm1) edge[0edge] (nm3)
(nm3) edge[0edge] (nm5)
(nm5) edge[0edge,densely dotted] (3)
(3) edge[0edge] (n)
(n) edge[0edge] (nm2)
(nm2) edge[0edge,densely dotted] (2)
(2) edge[0edge] (1)
(1) edge[0edge] (nm1)
(2) edge[0edge] (n);
\end{tikzpicture}
\end{figure}

\noindent The greatest common divisor of $n$ and $n/2$ is $n/2$. Thus $A_{01}^{n/2}$ has upper block-triangular form, with irreducible blocks of period $1$ along the diagonal. The first block has order $1$, and the remaining $n/2$ blocks order $2$. In particular, $A_{01}$ has a dominant eigenvalue strictly greater than $1$; since the vertex $e_0$ is fixed, the chain $A_{01}^t[\Delta]$ cannot shrink to a point. We have reached a contradiction and settled Case~(0h).

Of course Cases (0p) and (1p), as well as Cases (0h) and (1h), are exchanged by exchanging $0$ with $1$, and every $i\in\set{2,\ldots,n}$ with the corresponding $k_i$. Therefore the argument above applies to Case (1h) and establishes that $1$ must appear as  second-to-last node in $\Gcal_1^-$.
\end{proof}

The proof of Claim~\ref{ref10}, and thus of Theorem~\ref{ref6}, is almost complete. We have started from a contracting IFS $(A_0,A_1)$ such that $\Delta_0=N[\Delta]$ and $\Delta_1=L[\Delta]$, and have established that an appropriate conjugation reduces the pair of incomplete graphs $(\Gcal_0^-,\Gcal_1^-)$ (which determine the completed incidence graphs, and thus $A_0,A_1$ themselves) to one of the three possibilities (0p,1h), (0p,1p), (0h,1h) (the possibility (0h,1p) is equivalent to (0p,1h)). Moreover, we have proved that in the hyperbolic cases (0h) and (1h) the node chains are $1\ot 0\ot 2\ot 3\ot\cdots\ot n$ and $0\ot 1\ot k_2\ot k_3\ot\cdots\ot k_n$, respectively. 
Cases (0p) and (0h) correspond then to the M\"onkemeyer matrices $M_0$ and $M_0F$; if we establish that $k_i=i$ for every $i\in\set{2,\ldots,n}$, Cases (1p) and (1h) will correspond to $M_1$ and $M_1F$, thus settling Claim~\ref{ref10}.

As in the proof of Lemma~\ref{ref15}, no matter if we are in the parabolic or hyperbolic cases,
the $1$-chain $k_2\ot\cdots\ot k_n$ must run through the nodes $2,\ldots,n$ in some order. We claim that it does not contain backtracks. Indeed, if $i\to j$ were a long backtrack, that is with $i+1<j$,
then the argument in the proof of Lemma~\ref{ref15} would yield the existence of two incomparable loops at $j$, which is impossible. There cannot be two short backtracks $i\to i+1$ and $j\to j+1$, because this would create two loops for the word $10$ at $i$ and $j$, contrary to Lemma~\ref{ref14}(i). Finally, not even a single short backtrack $i\to i+1$ is possible. Indeed, direct inspection shows that, no matter if we are in the parabolic or hyperbolic cases, the incidence graph of $A_{10}$ would then be disconnected with an isolated loop at $i$, contrary to contractivity. 
Therefore there are no backtracks, and $k_i=i$ for every $i$ in $\set{2,\ldots,n}$.

\bibliography{bibliography}

\end{document}